\documentclass{amsart}

\usepackage[dvipdfmx]{graphicx}
\newtheorem{theorem}{Theorem}

\newtheorem{lemma}{Lemma}

\newtheorem{definition}{Definition}

\title{Operator topology \\ for logarithmic infinitesimal generators}
\author{Yoritaka Iwata}
\address[Y. Iwata]{Faculty of Chemistry, Materials and Bioengineering, Kansai University, Osaka 564-8680, Japan}
\email{ iwata$\_$phys@08.alumni.u-tokyo.ac.jp}

\begin{document}

\begin{abstract} 
Generally-unbounded infinitesimal generators are studied in the context of operator topology. 
Beginning with the definition of seminorm, the concept of locally convex topological vector space is introduced as well as the concept of Fr{\'e}chet space. 
These are the basic concepts for defining an operator topology.
Consequently, by associating the topological concepts with the convergence of sequence, a suitable mathematical framework for obtaining the logarithmic representation of infinitesimal generators is presented.
\end{abstract}



\maketitle

\section{Introduction} 
Let $X$ be an infinite/finite dimensional Banach space with the norm $\| \cdot \|$, and $Y$ be a dense subspace of $X$.
The Cauchy problem for abstract evolution equation of hyperbolic type \cite{70kato, 73kato} is defined by
\begin{equation} \begin{array}{ll} \label{eq01}
d u(t)/dt  - A(t) u(t) = f(t),  \qquad t \in [0,T], \vspace{2.5mm} \\
u(0) = u_0
\end{array} \end{equation}
in $X$, where $A(t): Y \to X$ is assumed to be the infinitesimal generator of the evolution operator $U(t,s)$ satisfying the strong continuity (for the definition of strong topology, refer to the following section) and the semigroup property:
\[ U(t,s) = U(t,r)U(r,s) \]
for $0 \le s \le r \le t < T$.
$U(t,s)$ is a two-parameter $C_0$-semigroup of operator that is a generalization of one-parameter $C_0$-semigroup and therefore an abstract generalization of the exponential function of operator.
For an  an infinitesimal generator $A(t)$ of $U(t,s)$, the solution $u(t)$ is represented by $u(t) = U(t,s) u_s$ with $u_s \in X$ for a certain $0 \le s \le T$ (cf. Hille-Yosida Theorem; for example see \cite{65yosida, 66kato, 79tanabe}).

\begin{figure*} [tb]
\begin{center}
\includegraphics[width=60mm]{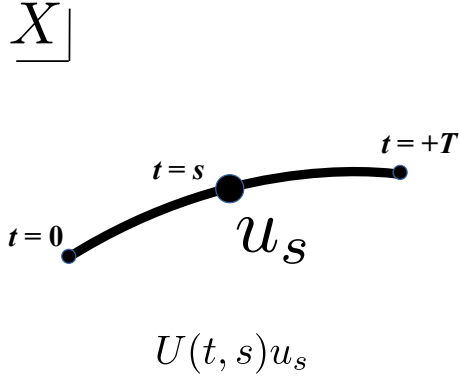} 
\caption{Trajectory $U(t,s) u_s$ in $X$. 
$U(t,s) \in B(X)$ is assumed to be strongly continuous with respect to time variables, so that a trajectory $U(t,s) u_s$ is continuous in $X$.
Note that it is necessary to replace the trajectory $U(t,s) u_s$ with the regularized trajectory $(e^{a(t,s)} - \kappa I) u_s$ to consider the negative time evolution \cite{17iwata-3, 19iwata}.
} \label{fig1}
\end{center}
\end{figure*}

\section{Operator topology} 
\subsection{The dual formalism of evolution equation}
The dual space of $X$ being denoted by $X^*$ is defined by
\[
X^* = L(X,K),
\]
where $K$ is a scalar field making up the space $X$, and $L(X,K)$ denotes the space of continuous linear functionals.
Since $K$ is also a Banach space, $L(X,K)$ satisfies the properties of Banach space.

Let $\langle \cdot, \cdot \rangle: X \times X^* \to \mathbb{C}$ be a dual product between $X$ and $X^*$.
The adjoint operator  $A(t)^*: D(A(t)) \to X^*$ is defined by the operator satisfying 
\[ \begin{array}{ll} 
\langle A(t) u , v \rangle  = \langle u, A(t)^* v \rangle
\end{array} \]
for any $u \in D(A(t))$ and $v \in D(A(t)^*)$.
If $X$ is a Hilbert space, the dual product is replaced with a scalar product $( \cdot, \cdot )$ equipped with $X$.
Unique dual correspondence is valid, if $X$ is strictly convex Banach space at least (for convex Banach space, see \cite{10brezis}).
By taking the dual product, the abstract evolution equation in $X$:
\begin{equation} \begin{array}{ll}  \label{eq01r}
d u(t)/dt  - A(t) u(t) = f(t),  \qquad t \in [0,T] 
\end{array} \end{equation}
is written as a scalar-valued evolution equation in $\mathbb{C}$:
\begin{equation} \begin{array}{ll} \label{eq02}
\langle d u(t)/dt, v(t) \rangle  - \langle A(t) u(t), v(t) \rangle = \langle f(t) , v(t) \rangle,  \qquad t \in [0,T]
\end{array} \end{equation}
for a certain $v(t) \in X^*$.
The formalism \eqref{eq02} has been considered by variational method of abstract evolution equations \cite{61lions, 72lions}, where the formalism \eqref{eq02} is also associated with the Gelfand triplets \cite{64gelfand}.
The equations \eqref{eq01r} and \eqref{eq02} cannot necessarily be equivalent in the sense of operator topology.

\subsection{locally-strong topology} 
Let $p(u)$ be a seminorm equipped with a space $\mathcal{X}$, and the family of seminorms be denoted by $P$.
Locally convex spaces are the generalization of normed spaces. 
Here the topology is called locally convex, if the topology admits a local base at 0 consisting of balanced, absorbent, convex sets. 
In other words, a topological space $\mathcal{X}$ is called locally convex, if its topology is generated by a family of seminorms satisfying 
\[ \begin{array}{ll}
\displaystyle \cap_{p \in P}  \left\{  u \in \mathcal{X}; ~ p(u) = 0  \right\}   = 0_{\mathcal{X}},
\end{array} \]
where $0_{\mathcal{X}}$ denotes the zero of topological space $\mathcal{X}$.
Fr{\'e}chet spaces are locally convex spaces that are completely metrizable with a certain complete metric. 
It follows that a Banach space $X$ is trivially a Fr{\'e}chet space.
For Banach spaces $X$ and $Y$, the bounded linear operators from $X$ to $Y$ is denoted by $B(X,Y)$.
In particular $B(X,X)$ is written by $B(X)$.
The operator space $B(X)$ is called the Banach algebra, since it holds the structure of algebraic ring.
The norm of $B(X)$, which means the operator norm, is defined by 
\[ \| T \|_{B(X)} = \sup_{ x \ne  0} \frac{\| Tx \|_X } { \| x\|_X}. \]
A norm is trivially a seminorm. 
Consequently, the topological space in this article is a Banach space $B(X)$.

There are several standard typologies defined on $B(X)$.
The topologies listed below are all locally convex, which implies that they are defined by a family of seminorms. 
The topologies are identified by the convergence arguments. 
Let $\{ T_n \}$ be a sequence in a Banach space $X$.
\begin{itemize}
\item $T_n \to T$ in the uniform topology, if the operator norm converges to 0;  \\
\item $T_n \to T$ in the strong topology, if $T_n x \to T x$ for any $x \in X$; \\
\item $T_n \to T$ in the weak topology, if $F(T_n x ) \to F(T x)$ for any $F \in X^*$ and $x \in X$; 
\end{itemize}
where the uniform topology is the strongest, and the weak topology is the weakest.
Indeed a topology is called stronger if it has more open sets and weaker if it has less open sets.
If $Y$ is a vector space of linear maps on the vector space $X$, then a topology $\sigma(X, Y)$ is defined to be the weakest topology on $X$ such that all elements of $Y$ are continuous.
The topology of $\sigma(X, Y)$ type is apparent if the formalism \eqref{eq02} is considered; the weak topology is written by $\sigma(B(X), B(X)^*)$.
Although there are some intermediate topologies between the above three; strong* operator topology, weak* operator topology, and so on, another type of topology is newly introduced in this article.

\begin{definition} [locally-strong topology] 
\quad \\
\begin{itemize}
\item $T_n \to T$ in the locally-strong topology, if $T_n \bar{x} \to T \bar{x} $ for a certain $\bar{x} \in X$. \\
\end{itemize}
\end{definition}
\hspace{-6mm}
This topology is utilized to define a weak differential appearing in the logarithmic representation.

\section{Infinitesimal generator}
\subsection{Logarithmic infinitesimal generator}


The logarithm of evolution operator is represented using the Riesz-Dunford integral.
A time interval $[0, T]$ with $0 \le s,t \le T$ is provided.
For a certain $u_s \in X$, let a trajectory $u(t) = U(t,s) u_s$ be given in a Banach space $X$.
For a given $U(t,s) \in B(X)$, its logarithm is well-defined \cite{17iwata-1}; there exists a certain complex number $\kappa$ satisfying
\[ \begin{array}{ll}
{\rm Log} (U(t,s)+\kappa I) = \frac{1}{2 \pi i} \int_{\Gamma} {\rm Log} \lambda  
 ~ ( \lambda  - \kappa - U(t,s) )^{-1}  d \lambda,
\end{array} \]
where an integral path $\Gamma$, which excludes the origin, is a circle in the resolvent set of $U(t,s) +\kappa I$.

Let us call ${\rm Log} (U(t,s)+\kappa I)$ the alternative infinitesimal generator to $A(t)$.
Since the alternative infinitesimal generator \cite{17iwata-3}
\[ a(t,s) := {\rm Log} (U(t,s)+\kappa I) \]
 is necessarily bounded on $X$, its exponential function $e^{a(t,s)}$ is always well defined.
Note that the alternative infinitesimal generator $a(t,s)$ is bounded on $X$, although the corresponding infinitesimal generator $A(t)$ is possibly an unbounded operator.
It follows that $e^{-a(t,s)} = (e^{a(t,s)})^{-1}$ is automatically well defined if $e^{a(t,s)}$ is well defined.
Also $e^{a(t,s)}$ is invertible regardless of the validity of the invertible property for original $U(t,s)$.
The logarithmic representation of infinitesimal generator (logarithmic infinitesimal generator, for short) is obtained as follows.

\begin{lemma} [Logarithmic infinitesimal generators \cite{17iwata-1}]
\label{thm1}
Let $t$ and $s$ satisfy $0 \le t,s \le T$, and $Y$ be a dense subspace of $X$.
If $A(t)$ and $U(t,s)$ commute, infinitesimal generators $\{ A(t) \}_{0 \le t \le T}$ are represented by means of the logarithm function; there exists a certain complex number $\kappa \ne 0$ such that
\begin{equation} \label{logex} \begin{array}{ll}
 A(t) ~ u =  (I - \kappa e^{-a(t,s)})^{-1}~  \partial_{t}  a(t,s)  ~ u, 
\end{array} \end{equation}
where $u$ is an element of a dense subspace $Y$ of $X$. \\
\end{lemma}

\begin{proof}
Only formal discussion is made here (for the detail, see \cite{17iwata-1, 19iwata}).   
\[
(U(t,s) + \kappa I )  \partial_{t}  a(t,s)   = (U(t,s) + \kappa I ) (U(t,s) + \kappa I )^{-1} \partial_t U(t,s) .
\]
It leads to
\[ \begin{array}{ll}
A(t) ~ u
 = U(t,s)^{-1} \partial_t U(t,s) u  
 = U(t,s)^{-1} (U(t,s) + \kappa I )  \partial_t a(t,s) ~ u  \vspace{1.5mm}  \\
  = (I + \kappa ( e^{a(t,s)}- \kappa I  )^{-1}  )  \partial_t a(t,s) ~ u  \vspace{1.5mm}  \\
 =  (  e^{a(t,s)}- \kappa I     + \kappa I  )   ( e^{a(t,s)}- \kappa I  )^{-1} \partial_{t}  a(t,s)  ~ u   \vspace{1.5mm}  \\
 =  ( I - \kappa e^{-a(t,s)} )^{-1} \partial_{t}  a(t,s)  ~ u ,
\end{array} \] 
where $u$ is an element in $Y$.
 \quad $\square$
\end{proof}

Equation (\ref{logex}) is the logarithmic representation of infinitesimal generator $A(t)$.
This representation is useful not only to mathematical analysis but also to operator algebra \cite{17iwata-2, 19iwata}.
In the next section the convergence of the limit in the differential operator $\partial_{t}$ of Eq. (\ref{logex}) is discussed.

\subsection{Differential operator in the logarithmic representation}

The convergence in the locally-strong topology is applied to the evolution operator $U(t,s) \in B(X)$.

\begin{definition} [Weak limit using the locally-strong topology]
For $0 \le t, s  \le T$, the weak limit\index{weak limit}
\[ \begin{array}{ll}
 \mathop{\rm lim}\limits_{h \to 0} h^{-1} (U(t+h,s) - U(t,s)) ~u_s 
= \mathop{\rm lim}\limits_{h \to 0}   h^{-1}(U(t+h,t) - I) ~ U(t,s) ~u_s,  
\end{array} \]
is assumed to exist for a certain $u_s$ in a dense subspace $Y$ of $X$. 
The limit ``$\lim$'' is practically denoted by ``$\mathop{\rm wlim}$'' in the following, since it is a limit in a kind of weak topology.

Let $t$-differential of $U(t,s)$ in a weak sense of the above be denoted by $\partial_t$, then it follows that
\begin{equation} \label{de-group} \begin{array}{ll} 
\partial_t U(t,s)~u_s = A(t) U(t,s) ~u_s,
 \end{array}  \end{equation}
and a generalized concept of infinitesimal generator $A(t): Y  ~\to~  X$ is introduced by
\begin{equation} \label{pe-group}
A(t) u_s := \mathop{\rm wlim}\limits_{h \to 0}  h^{-1} (U(t+h,t) - I) u_s
\end{equation}
for a certain $u_s \in Y$, where the convergence in $\mathop{\rm wlim}$ must be replaced with the strong convergence in the standard theory of abstract evolution equations \cite{79tanabe}. \\
 \end{definition}

The operator $A(t)$ defined in this way for a whole family $\{U(t,s)\}_{0 \le t,s \le T}$ is called the pre-infinitesimal generator in \cite{17iwata-1}, because only its exponetiability with a certain ideal domain is ensured without justifying the dense property of its domain space.
Indeed pre-infinitesimal generators are not necessarily infinitesimal generators, while infinitesimal generators are pre-infinitesimal generators.

\section{Main result}
According to the standard theory of abstract evolution equation \cite{79tanabe}, the evolution operator is assumed to be strongly continuous.
It follows that the trajectory $U(t,s) u_s$ is continuous in $X$.
Here is the reason why it is sufficient to consider the convergence of differential operator $\partial_t$ only with a fixed element ${\bar u} = u_s \in Y \subset X$ with $0 \le s \le T$.  
Also, in terms of analyzing the trajectory in finite/infinite dimensional dynamical systems, it is reasonable to consider the convergence in the topology unique to the trajectory. 
Consequently the infinitesimal generator can be extracted on a sample point in the interval (Fig.~\ref{fig1}).
Indeed, according to the independence between $t$ and $s$, 
\[
A(t) u_s = \mathop{\rm wlim}\limits_{h \to 0}  h^{-1} (U(t+h,t) - I) u_s
\]
is true for any $t \in [s, T]$, once $A(t)$ is obtained for a sample point $u_s \in Y$.
Such a restrictive topological treatment contributes to generalize the differential.

For a given evolution operator $U(t,s) \in B(X)$, the profile of locally-strong topology is obtained in this article.
The locally-strong topology is introduced for defining the logarithmic representation of infinitesimal generator.
In Banach space $B(X)$, a subset $F \subset B(X)$ is a closed set, if and only if
\[
\{ a_n \} \in F,~ a \in B(X),~ a_n \to a  \quad \Rightarrow \quad a \in F
\]
is satisfied $(n=1,2,\cdots)$, where the operation of limit depends on  a chosen topology.
Here the following two theorems are proved to clarify the mathematical property of the locally-strong topology.

\begin{theorem} \label{thm2}
The locally-strong topology is weaker than the strong topology.
\end{theorem}

\begin{proof}
It is enough to prove that a closed set in strong topology is closed in the locally-strong topology.
Let an arbitrary closed set of $B(X)$ in the strong topology be $V$.
It satisfies
\[
 \{ T_n \} \in V,~ T \in B(X),~  \lim_{n \to \infty} \| T_n x - T x \| = 0  
 \quad \Rightarrow \quad 
 T \in V
 \]
 for an arbitrary $x \in X$.
 In the locally-strong topology ($\| (T_n - T) \bar{x} \|$ for a certain ${\bar x} \in X$), the convergence $T_n \to T \in V$ is true. 
 \quad $\square$
\end{proof}

\begin{theorem} \label{thm3}
The locally-strong topology is not necessarily stronger than the weak topology.
\end{theorem}

\begin{proof}
The proof is carried out in the similar manner to Theorem \ref{thm2}.
Let an arbitrary closed set of $B(X)$  in the locally-strong topology be $V$.
It satisfies
\begin{equation} \label{stt}
 \{ T_n \} \in V,~ T \in B(X),~    \lim_{n \to \infty} \| T_n  {\bar x}- T {\bar x} \| = 0  
 \quad \Rightarrow \quad 
 T \in V
 \end{equation}
 for a fixed ${\bar x} \in X$.
 By taking the dual product of an arbitrary $F \in X^*$, it follows that
 \[
 \{ T_n \} \in V,~ a \in B(X),~    \lim_{n \to \infty} \langle (T_n - T) {\bar x},~ F \rangle = 0  
 \quad \Rightarrow \quad 
 T \in V
 \]
It shows that the closedness of $V$ in a locally-weak topology being defined by fixing $x = {\bar x}$.

On the other hand, weak convergence cannot be assured if $x \ne {\bar x}$.
Indeed, for $x_1 \ne {\bar x}$, the statement 
\[
 \{ T_n \} \in V,~ a \in B(X),~    \lim_{n \to \infty} \langle (T_n - T) x_1,~ F \rangle = 0  
 \quad \Rightarrow \quad 
 T \in V
 \]
does not follow from the statement \eqref{stt}.
There is no guarantee for locally-strong topology to be stronger than the weak topology.
  \quad $\square$
\end{proof}

\section{Summary}
The concept of locally-strong topology is introduced by the proofs clarifying its specific topological weakness.
The locally-strong topology is a topology unique to the solution trajectory of abstract evolution equations (Fig.~\ref{fig1}).
Although the locally-strong topology has already been utilized to clarify the algebraic structure of semigroups of operators and their infinitesimal generators, it is also expected to be useful to analyze each single trajectory defined in finite/infinite-dimensional dynamical systems. 

\section*{Acknowledgments}
The author is grateful to Prof. Emeritus Hiroki Tanabe for valuable comments.
This work was partially supported by JSPS KAKENHI Grant No. 17K05440.

	%
	%
	%
	
	
	


\end{document}